\documentclass[11pt]{article}
\usepackage{amssymb}
\usepackage{amsmath}
\usepackage{graphicx}
\usepackage[english]{babel}
\usepackage{amsfonts}
\usepackage{color}

\usepackage{epsfig}
\usepackage{tikz}
\usepackage{amsmath,amssymb,amsfonts}
\usepackage{amsfonts}
\usepackage{amscd}
\usepackage{latexsym}
\usepackage{tikz,epsf}
\usepackage{tikz,amsmath}
\tikzset{
    partial ellipse/.style args={#1:#2:#3}{
    insert path={+ (#1:#3) arc (#1:#2:#3)}
    }
}

\newcommand{\vertex}{\node[vertex]}

\tikzstyle{vertex}=[circle, draw, inner sep=0pt, minimum size=4pt]
\tikzstyle{vertexa}=[circle, draw, inner sep=0pt, minimum size=2pt]

\input epsf
\providecommand{\U}[1]{\protect\rule{.1in}{.1in}}

\newtheorem{theorem}{Theorem}

\newtheorem{corollary}[theorem]{Corollary}

\newtheorem{example}[theorem]{Example}

\newtheorem{lemma}[theorem]{Lemma}

\newtheorem{remark}[theorem]{Remark}

\newenvironment{proof}[1][Proof]{\textbf{#1.} }{\small$\square$}
\begin{document}

\title{The Laplacian eigenvalue $2$ of bicyclic graphs}
\author{{\small Doost Ali Mojdeh\thanks{Corresponding author}, {\small Mohammad Habibi{$^\dag$}}}, {\small Masoumeh Farkhondeh{$^\ddag$}}\\ {\small {$^*$}Department of Mathematics,}
\\{\small University of Mazandaran, Babolsar, Iran}\\ {\small email: {$^*$}damojdeh@umz.ac.ir}\\
{\small {$^{\dag \ddag}$}Department of Mathematics}\\
{\small Tafresh University, Tafresh, Iran}\\
{\small email: {$^\dag$}mhabibi@tafreshu.ac.ir,}\\
{\small email: {$^\ddag$}mfarkhondeh81@gmail.com }}
\date{}
\maketitle

\begin{abstract} If $G$ is a graph, its Laplacian is the
difference between diagonal matrix of its vertex degrees and its
adjacency matrix. A one-edge connection of two graphs $G_{1}$ and
$G_{2}$ is a graph $G=G_{1}\odot G_{2}$ with $V(G)=V(G_{1})\cup
V(G_{2})$ and $E(G)= E(G_{1})\cup E(G_{2})\cup \{e=uv\}$ where $u\in
V(G_1)$ and $v\in V(G_2)$. In this paper, we consider the
eigenvector of unicycle graphs.  We study the relationship between
the Laplacian eigenvalue $2$ of unicyclic graphs $G_1$ and $G_2$;
and bicyclic graphs $G=G_{1}\odot G_{2}$. We also characterize the
broken sun graphs and the one edge connection of two broken sun
graphs by their Laplacian eigenvalue $2$.
\end{abstract}

{\textbf  MSC 2010:} Primary 05C50; Secondary 15A18

{\textbf Keywords}: Laplacian eigenvalue, multiplicity, eigenvector, unicyclic
graph, bicyclic graph

\section{Introduction}
\indent All graphs in this paper are finite and undirected with no
loops or multiple edges. Let $G$ be a graph. The vertex set and the
edge set of $G$ are denoted by $V(G)$ and $E(G)$, respectively. The
\textit{Laplacian matrix} of $G$ is $L(G)=D(G)-A(G)$, where
$D(G)=diag(d(v_{1}),\ldots , d(v_{n}))$ is a diagonal matrix and
$d(v)$ denotes the degree of the vertex $v$ in $G$ and $A(G)$ is the
adjacency matrix of $G$. Denoting its eigenvalues by $\mu_{1}(G)\geq
\cdots \geq \mu_{n}(G)=0$, we shall use the notation $\mu_k(G)$ to
denote the $k$th  Laplacian eigenvalue of the graph G. Also, the
multiplicity of the eigenvalue $\mu$ of $L(G)$ is denoted by
$m_{G}(\mu)$. For any $v\in V(G)$, let $N(v)$ be the set of all
vertices adjacent to $v$. A vertex of degree one is called a
\textit{pendant} vertex and a vertex is said \textit{quasi pendant}
 (support vertex) if it is incident to a pendant
vertex. A \textit{matching} of $G$ is a set of pairwise disjoint
edges of $G$. The \textit{matching number} of $G$ is a maximum of
cardinalities of all matchings of $G$, denoted by $\alpha^{'}(G)$.
Clearly, $n\geq \alpha^{'}(G)$. In particular, if
$n=2\alpha^{'}(G)$, then $G$ has a perfect matching.\\
\indent Connected graphs in which the number of edges equals the
number of vertices are called \textit{unicyclic graphs}. Therefore,
a unicyclic graph is either a cycle or a cycle with some attached
trees. Let $\mathfrak{U}_{n,g}$ be the set of all unicyclic graphs
of order $n$ with girth $g$. A \textit{rooted tree} is a tree in
which one vertex has been designated the root. Furthermore, assume
$T_{i}$ is a rooted tree of order $n_{i}\geq 1$ attached to
$v_{i}\in V(T_{i})\cap V(C_{g})$, where $\sum_{i=1}^{g}n_{i}=n$.
This unicyclic graph is denoted by $C(T_{1},...,T_{g})$. The
\textit{sun graph} of order $2n$ is a cycle $C_{n}$ with an edge
terminating in a pendant vertex attached to each vertex. A
\textit{broken sun graph} is a unicyclic subgraph of a sun graph.
Let $n_{i}(G)$ be the number of vertices of degree $i$ in $G$.

A \textit{ one-edge connection} of two graphs $G_{1}$ and $G_{2}$
is a graph $G$ with $V(G)=V(G_{1})\cup V(G_{2})$ and $E(G)=
E(G_{1})\cup E(G_{2})\cup \{e=uv\}$ where $u\in V(G_1)$ and $v\in
V(G_2)$ is denoted by $G=G_{1}\odot G_{2}$. If
$G_{1}=C(T_1,\ldots,T_{g_1})$ and $G_{2}=C(T'_1,\ldots,T'_{g_2})$
are unicyclic graphs, then $G=G_{1}\odot G_{2}$ is a bicyclic graph
and denoted  by $G=C(T_1,\ldots,T_{g_1},T'_1,\ldots,T'_{g_2})$. In
this manuscript, we would like to study the eigenvalue 2 in bicyclic
graphs. We provide a necessary and sufficient condition under which
a bicyclic graph with a perfect matching has a Laplacian eigenvalue
2. For more about Laplacian of some parameters of graphs  we refer
\cite{akbari, kiani1, kiani2, mirzakhah}.\par

\section{Preliminary results}
\indent By \cite[Theorem 13]{mieghem} due to Kelmans and Chelnokov,
the Laplacian coefficient, $\xi_{n-k}$, can be expressed in terms of
subtree structures of $G$, for $0\leq k\leq n$. Suppose that $F$ is
a spanning forest of $G$ with components $T_{i}$ of order $n_{i}$,
and $\gamma(F)=\Pi_{i=1}^{k}n_{i}$. The \textit{Laplacian
characteristic polynomial} of $G$ is denoted by $ L_{G}(x)=
det(xI-L(G)) = \Sigma_{i=0}^{n}(-1)^{i}\xi_{i}x^{n-i}$.
\begin{theorem}\label{component} \emph{\cite[Theorem~7.5]{biggs}}
The Laplacian coefficient $\xi_{n-k}$ of a graph $G$ of order $n$ is
given by $\xi_{n-k}=\sum_{F\in \mathfrak{F}_{k}}\gamma(F)$, where
$\mathfrak{F}_{k}$ is the set of all spanning forest of $G$ with
exactly $k$ components.
\end{theorem}

In particular, we have $\xi_{0}=1$, $\xi_{1}=2m$, $\xi_{n}=0$, and
$\xi_{n-1}=n\tau(G)$, in which $\tau(G)$ denotes the number of
spanning trees of $G$.

Let $G$ be a graph with $n$ vertices. It is convenient to adopt the
following terminology from \cite{fiedler}: for a vector
$X=(x_{1},\ldots,x_{n})^t\in \mathbb{R}^{n}$, we say $X$ gives a
valuation of the vertex of $V$, and with each vertex $v_{i}$ of $V$,
we associate the number $x_{i}$, which is the value of the vertex
$v_{i}$, that is $x(v_{i})=x_{i}$. Then $\mu$ is an eigenvalue of
$L(G)$ with the corresponding eigenvector $X=(x_{1},\ldots,x_{n})$
if and only if $X\neq 0$ and
\begin{equation}\label{one}
(d(v_{i})-\mu)x_{i}=\sum_{v_{j}\in N(v_{i})} x_{j}\,\,
\mathrm{for}\,\,\mathrm{all}\,\, i=1,\ldots,n.
\end{equation}

\indent It has been shown that if $T$ is a tree containing a perfect
matching, then $L(T)$ has an eigenvalue $2$ and
$\mu_{\alpha^{'}(T)}(T)=\mu_{n/2}(T)=2$, \cite[Theorem~2]{ming}. In
\cite[Theorem 2]{fan} the authors proved that, if $T$ is a tree with
a perfect matching, vector $X\neq 0$ is an eigenvector of $L(T)$
corresponding to the eigenvalue $2$ if and only if $X$ has exactly
two distinct entries $-t$ and $t$, for each edge $e=uv\in M$,
$x(u)=-x(v)$ and for each edge $e=uv\notin M$, $x(u)=x(v)$.


\section{The eigenvector of Laplacian eigenvalue $2$}
In the follow,  we study some results on broken sun graphs and
unicycle graphs. Then we establish the eigenvector of these types of
graphs that have Laplacian eigenvalue $2$.   First we see a theorem
from \cite{mohar}.
\begin{theorem}\label{eqution} \emph{\cite[Theorem 3.2]{mohar}}
Let $G$ be a graph on $n$ vertices and $e$ be an edge of $G$. Let
$\mu_{i}(G),\ (1\le i\le n)$ be the eigenvalue of $L(G)$. Then the
following holds: $$\mu_{1}(G)\geq \mu_{1}(G-e)\geq \mu_{2}(G)\geq
\mu_{2}(G-e)\geq\cdots \geq \mu_{n}(G)\geq \mu_{n}(G-e)$$
\end{theorem}


The following remark remember us some notes which are useful for the
next results.

\begin{remark}\label{treecycle}
Let $T$ be a tree of order $n\geq 3$ with a perfect matching, and
let $X$ be an eigenvector of $T$ corresponding to the eigenvalue
$2$. Then by \cite[Theorem~2]{fan}, $T$ has $\frac{n}{2}$ vertices
with value $1$ and $\frac{n}{2}$ vertices with value $-1$ given by
$X$. Let $V_1$ and $V_2$ be the sets of the former $\frac{n}{2}$ and
the latter $\frac{n}{2}$ vertices, respectively. By
\cite[Theorem~3.1]{merris}, if we add edges between any two
non-incident vertices in $V_1$ or $V_2$, then $2$ is also the
eigenvalue of the result graph. Hence, if $u$ and $v$ belong to
$V_1$ (or $V_2$) , then $L(G=T\cup \{uv\})$ has an eigenvalue $2$
and $X=(x_{1},\ldots ,x_{n})^{t}$ is an eigenvector of $L(G)$
corresponding to the eigenvalue $2$ where $x_{i}\in \{-1,1\}$.
\end{remark}



\begin{theorem}\label{seigenv}
Let $G$ be a broken sun graph containing a perfect matching which
has a Laplacian eigenvalue $2$. Then there exists an eigenvector
corresponding to the eigenvalue $2$ such as $X=(x_{1},\ldots
,x_{n})^t$ such that $x_{i}\in \{-r,r\}$.
\end{theorem}

\begin{proof}
By induction on $g$ and using Remark \ref{treecycle}, we prove that
$x(u)=x(v)$ where $X$ is an eigenvector of $L(G)$ corresponding to
the eigenvalue $2$. Assume that $M$ is a perfect matching in $G$.
The following figures ($1-4$) for all broken sun graphs with $3\leq
g\leq 6$, that containing a perfect matching and having Laplacian
eigenvalue $2$, for each arbitrary edge $e=uv\notin M$, by removing
$e$, we have a tree $T=G-e$ with a perfect matching. Thus assume
that $X=(x_{1},\ldots ,x_{n})^t$ be an eigenvector of $L(T)$
corresponding to the eigenvalue $2$ such that $x_{i}\in \{-r,r\}$,
by \cite[Theorem~2]{fan} and $x_T(u)=x_T(v)$. Also $X$ is an
eigenvector of $L(G)$ corresponding to the eigenvalue $2$, by
\ref{treecycle}.

\begin{figure}[!ht]
  \[\begin{tikzpicture}[scale=.6,thin]
 \vertex[fill] (v1) at (3,2)[] {};
 \vertex[fill] (v2) at (3,3) [] {};
 \vertex[fill] (v3) at (2,1) [label=south:$u$] {};
 \vertex[fill] (v4) at (4,1)[label=south:$v$] {};
 \vertex[fill] (v5) at (1,1) [] {};
 \vertex[fill](v6) at (5,1) [] {};
 \vertex[fill](v7) at (9,2) [label=right:$r$] {};
 \vertex[fill](v8) at (9,3) [label=right:$-r$] {};
 \vertex[fill](v9) at (8,1)[label=south:$r$] {};
 \vertex[fill](v10) at (10,1) [label=south:$r$] {};
 \vertex[fill](v11) at (7,1) [label=south:$-r$] {};
 \vertex[fill](v12) at (11,1) [label=south:$-r$] {};

 \path
 (v3) edge (v4)
 (v4) edge (v1)
 (v1) edge (v3)
 (v10) edge (v7)
 (v9) edge (v7);

 \draw[line width=2pt](1,1)  .. controls (1.5,1) ..(2,1);
 \draw[line width=2pt](3,2)  .. controls (3,2.5) ..(3,3);
 \draw[line width=2pt](4,1)  .. controls (4.5,1) ..(5,1);
 \draw[line width=2pt](9,2)  .. controls (9,2.5) ..(9,3);
 \draw[line width=2pt](7,1)  .. controls (7.5,1) ..(8,1);
 \draw[line width=2pt](10,1)  .. controls (10.5,1) ..(11,1);
 \draw[line width=2pt] (6,1.5).. controls(6,1.5)..node[pos=0.2,anchor=south]{$\Longrightarrow$}(6,1.5);
\end{tikzpicture} \]
\caption{g=3, $x_T(u)=x_T(v)$}
\end{figure}
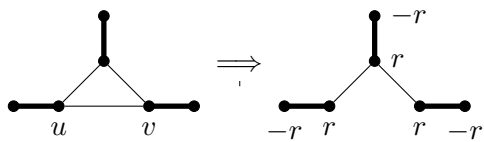

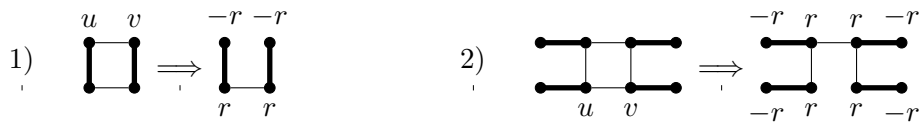
\begin{figure}[!ht]
  \[\begin{tikzpicture}[scale=.6,thin]
 \vertex[fill] (v1) at (-5,1)[] {};
 \vertex[fill] (v2) at (-4,1) [] {};
 \vertex[fill] (v3) at (-4,2) [label=north:$v$] {};
 \vertex[fill] (v4) at (-5,2)[label=north:$u$] {};
 \vertex[fill] (v5) at (-2,1) [label=south:$r$] {};
 \vertex[fill](v6) at (-1,1) [label=south:$r$] {};
 \vertex[fill](v7) at (-1,2) [label=north:$-r$] {};
 \vertex[fill](v8) at (-2,2) [label=north:$-r$] {};

 \vertex[fill](v9) at (5,2)[] {};
 \vertex[fill](v10) at (6,2) [] {};
 \vertex[fill](v11) at (6,1) [label=south:$u$] {};
 \vertex[fill](v12) at (7,2) [] {};
 \vertex[fill](v13) at (5,1) [] {};
 \vertex[fill](v14) at (7,1) [label=south:$v$] {};
 \vertex[fill](v15) at (8,1) [] {};
 \vertex[fill](v16) at (8,2) [] {};
 \vertex[fill](v17) at (10,2) [label=north:$-r$] {};
 \vertex[fill](v18) at (11,2) [label=north:$r$] {};
 \vertex[fill](v19) at (11,1) [label=south:$r$] {};
 \vertex[fill](v20) at (10,1) [label=south:$-r$] {};
 \vertex[fill](v21) at (12,1) [label=south:$r$] {};
 \vertex[fill](v22) at (13,1) [label=south:$-r$] {};
 \vertex[fill](v23) at (12,2) [label=north:$r$] {};
 \vertex[fill](v24) at (13,2) [label=north:$-r$] {};

 \path
 (v1) edge (v2)
 (v3) edge (v4)
 (v5) edge (v6)
 (v10) edge (v11)
 (v11) edge (v14)
 (v14) edge (v12)
 (v10) edge (v12)
 (v19) edge (v18)
 (v23) edge (v18)
 (v23) edge (v21);

   \draw[line width=2pt] (-4,2) .. controls (-4,1.5) ..(-4,1);
   \draw[line width=2pt] (-5,1) .. controls (-5,1.5) ..(-5,2);
   \draw[line width=2pt] (-1,2) .. controls (-1,1.5) ..(-1,1);
   \draw[line width=2pt] (-2,1) .. controls (-2,1.5) ..(-2,2);

   \draw[line width=2pt] (5,2) .. controls (5.5,2) ..(6,2);
   \draw[line width=2pt] (5,1) .. controls (5.5,1) ..(6,1);
   \draw[line width=2pt] (7,1) .. controls (7.5,1) ..(8,1);
   \draw[line width=2pt] (7,2) .. controls (7.5,2) ..(8,2);

   \draw[line width=2pt] (10,2).. controls (10.5,2) ..(11,2);
   \draw[line width=2pt] (10,1) .. controls (10.5,1) ..(11,1);
   \draw[line width=2pt] (12,1) .. controls (12.5,1) ..(13,1);
   \draw[line width=2pt] (12,2) .. controls (12.5,2) ..(13,2);
   \draw[line width=2pt] (-6.5,1) ..controls(-6.5,1)..node[pos=0.2,anchor=south]{$1)$}(-6.5,1);

  \draw[line width=2pt] (3.5,1) ..controls(3.5,1)..node[pos=0.2,anchor=south]{$2)$}(3.5,1);

  \draw[line width=2pt] (-3,1).. controls(-3,1)..node[pos=0.2,anchor=south]{$\Longrightarrow$}(-3,1);
  \draw[line width=2pt] (9,1)..controls(9,1)..node[pos=0.2,anchor=south]{$\Longrightarrow$}(9,1);
\end{tikzpicture} \]
\caption{g=4, $x_T(u)=x_T(v)$}
\end{figure}

\begin{figure}[!ht]
  \[\begin{tikzpicture}[scale=.5,thin]
 \vertex[fill] (v1) at (-5.5,0)[] {};
 \vertex[fill] (v2) at (-5,-1) [label=south:$v'$] {};
 \vertex[fill] (v3) at (-3,-1)[label=south:$u'$] {};
 \vertex[fill] (v4) at(-2.5,0)[label=north:$v$] {};
 \vertex[fill] (v5) at (-4,1) [label=right:$u$] {};
 \vertex[fill](v6) at (-4,2) [] {};
 \vertex[fill](v7) at (1.5,2) [label=north:$-r$] {};
 \vertex[fill](v8) at (2,1) [label=left:$r$] {};
 \vertex[fill](v9) at (4,1) [label=right:$r$] {};
 \vertex[fill](v10) at (4.5,2) [label=right:$-r$] {};
 \vertex[fill](v11) at (3,3) [label=right:$-r$] {};
 \vertex[fill](v12) at (3,4) [label=right:$r$] {};
 \vertex[fill](v13) at (1.5,-2) [label=north:$r$] {};
 \vertex[fill](v14) at (2,-3) [label=left:$-r$] {};
 \vertex[fill](v15) at (4,-3) [label=right:$-r$] {};
 \vertex[fill](v16) at (4.5,-2) [label=right:$r$] {};
 \vertex[fill](v17) at (3,-1) [label=right:$r$] {};
 \vertex[fill](v18) at (3,0) [label=right:$-r$] {};
 \vertex[fill](v19) at (8.5,0) [] {};
 \vertex[fill](v20) at (9,-1) [] {};
 \vertex[fill](v21) at (11,-1) [] {};
 \vertex[fill](v22) at (11.5,0) [label=north:$v$]  {};
 \vertex[fill](v23) at (10,1) [label=right:$u$] {};
 \vertex[fill](v24) at (7.5,0.5) [] {};
 \vertex[fill](v25) at (9,-2) [] {};
 \vertex[fill](v26) at (11,-2) [] {};
 \vertex[fill](v27) at (12.5,0.5) [] {};
 \vertex[fill](v28) at (10,2) [] {};
 \vertex[fill](v29) at (14.5,0) [label=north:$r$] {};
 \vertex[fill](v30) at (15,-1) [label=left:$r$] {};
 \vertex[fill](v31) at (17,-1) [label=right:$r$] {};
 \vertex[fill](v32) at (17.5,0) [label=north:$r$] {};
 \vertex[fill](v33) at (16,1) [label=right:$r$] {};
 \vertex[fill](v34) at (13.5,0.5) [label=north:$-r$] {};
 \vertex[fill](v35) at (15,-2) [label=left:$-r$] {};
 \vertex[fill](v36) at (17,-2) [label=right:$-r$] {};
 \vertex[fill](v37) at (18.5,0.5) [label=north:$-r$] {};
 \vertex[fill](v38) at (16,2) [label=right:$-r$] {};

 \path
 (v2) edge (v3)
 (v4) edge (v5)
 (v5) edge (v1)
 (v8) edge (v9)
 (v11) edge (v7)
 (v16) edge (v17)
 (v17) edge (v13)
 (v19) edge (v20)
 (v20) edge (v21)
 (v21) edge (v22)
 (v22) edge (v23)
 (v23) edge (v19)
 (v29) edge (v30)
 (v30) edge (v31)
 (v31) edge (v32)
 (v33) edge (v29);

   \draw[line width=2pt] (-5.5,0).. controls (-5.25,-0.5) ..(-5,-1);
   \draw[line width=2pt] (-3,-1).. controls (-2.75,-0.5) ..(-2.5,0);
   \draw[line width=2pt] (-4,1).. controls (-4,1.5) ..(-4,2);
   \draw[line width=2pt] (1.5,2).. controls (1.75,1.5) ..(2,1);
   \draw[line width=2pt] (4,1) .. controls (4.25,1.5) ..(4.5,2);
   \draw[line width=2pt] (3,3).. controls (3,3.5) ..(3,4);
   \draw[line width=2pt] (1.5,-2) .. controls (1.75,-2.5) ..(2,-3);
   \draw[line width=2pt] (4,-3) .. controls (4.25,-2.5) ..(4.5,-2);
   \draw[line width=2pt] (3,-1).. controls (3,-0.5) ..(3,0);
   \draw[line width=2pt] (8.5,0)  .. controls (8,0.25) ..(7.5,0.5);
   \draw[line width=2pt] (10,1)  .. controls (10,1.5) ..(10,2);
   \draw[line width=2pt] (9,-1)  .. controls (9,-1.5) ..(9,-2);
   \draw[line width=2pt] (11,-1)  .. controls (11,-1.5) ..(11,-2);
   \draw[line width=2pt] (11.5,0)  .. controls (12,0.25) ..(12.5,0.5);
   \draw[line width=2pt] (16,2)  .. controls (16,1.5) ..(16,1);
   \draw[line width=2pt] (14.5,0)  .. controls (14,0.25) ..(13.5,0.5);
   \draw[line width=2pt] (15,-1)  .. controls (15,-1.5) ..(15,-2);
   \draw[line width=2pt] (17,-1)  .. controls (17,-1.5) ..(17,-2);
   \draw[line width=2pt] (17.5,0)  .. controls (18,0.25) ..(18.5,0.5);
 \draw[line width=2pt] (-0.5,1.5) ..controls(-0.5,1.5)..node[pos=0.2,anchor=south]{$i)$}(-0.5,1.5);

 \draw[line width=2pt] (-0.5,-2.5) .. controls (-0.5,-2.5)..node[pos=0.2,anchor=south]{$ii)$} (-0.5,-2.5);

 \draw[line width=2pt] (-6.5,-0.5) ..controls(-6.5,-0.5)..node[pos=0.2,anchor=south]{$1)$}(-6.5,-0.5);

 \draw[line width=2pt] (6.5,-0.5) ..controls(6.5,-0.5)..node[pos=0.2,anchor=south]{$2)$}(6.5,-0.5);

 \draw[line width=2pt] (-1.5,-0.5).. controls(-1.5,-0.5)..node[pos=0.2,anchor=south]{$\Longrightarrow$}(-1.5,-0.5);
 \draw[line width=2pt] (13,-0.5)..controls(13,-0.5)..node[pos=0.2,anchor=south]{$\Longrightarrow$}(13,-0.5);
\end{tikzpicture} \] \caption{g=5, $x_T(u)=x_T(v)$ and $ x_{T'}(u')=x_{T'}(v')$}
\end{figure}
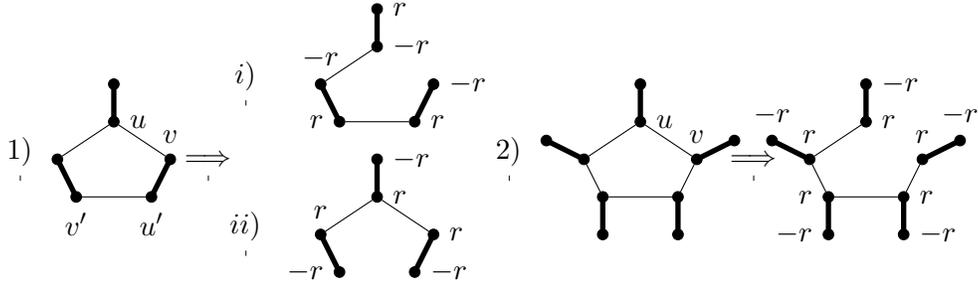

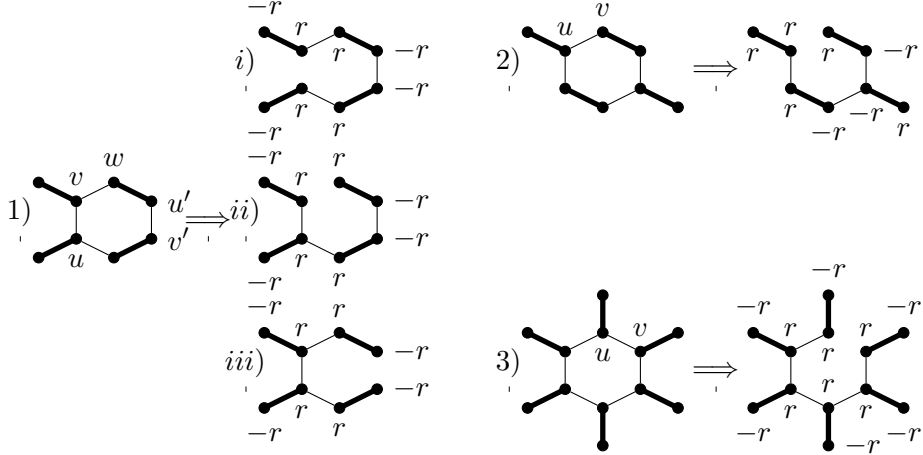
\begin{figure}[!ht]
  \[\begin{tikzpicture}[scale=.5,thin]
 \vertex[fill] (v1) at (-5,4)[label=north:$v$] {};
 \vertex[fill] (v2) at (-5,3) [label=south:$u$] {};
 \vertex[fill] (v3) at (-4,2.5) [] {};
 \vertex[fill] (v4) at(-3,3)[label=right:$v'$]  {};
 \vertex[fill] (v5) at (-3,4) [label=right:$u'$] {};
 \vertex[fill](v6) at (-4,4.5) [label=north:$w$] {};
 \vertex[fill](v7) at (-6,4.5) [] {};
 \vertex[fill](v8) at (-6,2.5) [] {};
 \vertex[fill](v9) at (3,8) [label=right:$-r$] {};
 \vertex[fill](v10) at (3,7) [label=right:$-r$] {};
 \vertex[fill](v11) at (1,7) [label=south:$r$] {};
 \vertex[fill](v12) at (2,6.5) [label=south:$r$] {};
 \vertex[fill](v13) at (2,8.5) [label=south:$r$] {};
 \vertex[fill](v14) at (1,8) [label=north:$r$] {};
 \vertex[fill](v15) at (0,8.5) [label=north:$-r$] {};
 \vertex[fill](v16) at (0,6.5) [label=south:$-r$] {};
 \vertex[fill](v17) at (1,4) [label=north:$r$] {};
 \vertex[fill](v18) at (1,3) [label=south:$r$] {};
 \vertex[fill](v19) at (2,2.5) [label=south:$r$] {};
 \vertex[fill](v20) at (3,3) [label=right:$-r$] {};
 \vertex[fill](v21) at (3,4) [label=right:$-r$] {};
 \vertex[fill](v22) at (2,4.5) [label=north:$r$] {};
 \vertex[fill](v23) at (0,4.5)[ label=north:$-r$] {};
 \vertex[fill](v24) at (0,2.5) [label=south:$-r$] {};
 \vertex[fill](v25) at (1,0) [label=north:$r$] {};
 \vertex[fill](v26) at (1,-1) [label=south:$r$] {};
 \vertex[fill](v27) at (2,-1.5) [label=south:$r$] {};
 \vertex[fill](v28) at (3,-1) [label=right:$-r$] {};
 \vertex[fill](v29) at (3,0) [label=right:$-r$] {};
 \vertex[fill](v30) at (2,0.5) [label=north:$r$] {};
 \vertex[fill](v31) at (0,0.5) [label=north:$-r$] {};
 \vertex[fill](v32) at (0,-1.5) [label=south:$-r$] {};
 \vertex[fill](v33) at (8,8) [label=north:$u$] {};
 \vertex[fill](v34) at (8,7) [] {};
 \vertex[fill](v35) at (9,6.5) [] {};
 \vertex[fill](v36) at (10,7) [] {};
 \vertex[fill](v37) at (10,8) [] {};
 \vertex[fill](v38) at (9,8.5) [label=north:$v$] {};
 \vertex[fill](v39) at (7,8.5) [] {};
 \vertex[fill](v40) at (11,6.5) [] {};
 \vertex[fill](v41) at (14,8) [label=north:$r$] {};
 \vertex[fill](v42) at (14,7) [label=south:$r$] {};
 \vertex[fill](v43) at (15,6.5) [label=south:$-r$] {};
 \vertex[fill](v44) at (16,7) [label=south:$-r$] {};
 \vertex[fill](v45) at (16,8) [label=right:$-r$] {};
 \vertex[fill](v46) at (15,8.5) [label=south:$r$] {};
 \vertex[fill](v47) at (13,8.5) [label=south:$r$] {};
 \vertex[fill](v48) at (17,6.5) [label=south:$r$] {};
 \vertex[fill](v49) at (8,0) [] {};
 \vertex[fill](v50) at (8,-1) [] {};
 \vertex[fill](v51) at (9,-1.5) [] {};
 \vertex[fill](v52) at (10,-1) [] {};
 \vertex[fill](v53) at (10,0) [label=north:$v$] {};
 \vertex[fill](v54) at (9,0.5) [label=south:$u$] {};
 \vertex[fill](v55) at (7,0.5) [] {};
 \vertex[fill](v56) at (7,-1.5) [] {};
 \vertex[fill](v57) at (9,-2.5) [] {};
 \vertex[fill](v58) at (11,-1.5) [] {};
 \vertex[fill](v59) at (11,0.5) [] {};
 \vertex[fill](v60) at (9,1.5) [] {};
 \vertex[fill](v61) at (14,0) [label=north:$r$] {};
 \vertex[fill](v62) at (14,-1) [label=south:$r$] {};
 \vertex[fill](v63) at (15,-1.5) [label=north:$r$] {};
 \vertex[fill](v64) at (16,-1) [label=south:$r$] {};
 \vertex[fill](v65) at (16,0) [label=north:$r$] {};
 \vertex[fill](v66) at (15,0.5) [label=south:$r$] {};
 \vertex[fill](v67) at (13,0.5) [label=north:$-r$] {};
 \vertex[fill](v68) at (13,-1.5) [label=south:$-r$] {};
 \vertex[fill](v69) at (15,-2.5) [label=right:$-r$] {};
 \vertex[fill](v70) at (17,-1.5) [label=south:$-r$] {};
 \vertex[fill](v71) at (17,0.5) [label=north:$-r$] {};
 \vertex[fill](v72) at (15,1.5) [label=north:$-r$] {};

 \path
 (v1) edge (v2)
 (v2) edge (v3)
 (v4) edge (v5)
 (v6) edge (v1)
 (v11) edge (v12)
 (v13) edge (v14)
 (v10) edge (v9)
 (v17) edge (v18)
 (v18) edge (v19)
 (v20) edge (v21)
 (v25) edge (v26)
 (v26) edge (v27)
 (v30) edge (v25)
 (v33) edge (v34)
 (v35) edge (v36)
 (v36) edge (v37)
 (v38) edge (v33)
 (v41) edge (v42)
 (v43) edge (v44)
 (v44) edge (v45)
 (v49) edge (v50)
 (v50) edge (v51)
 (v51) edge (v52)
 (v52) edge (v53)
 (v53) edge (v54)
 (v54) edge (v49)
 (v61) edge (v62)
 (v62) edge (v63)
 (v63) edge (v64)
 (v64) edge (v65)
 (v66) edge (v61);

 \draw[line width=2pt] (-6,2.5) .. controls (-5.5,2.75) .. (-5,3);
 \draw[line width=2pt] (-6,4.5) .. controls (-5.5,4.25) .. (-5,4);
 \draw[line width=2pt] (-4,2.5) .. controls (-3.5,2.75) .. (-3,3);
 \draw[line width=2pt] (-3,4) .. controls (-3.5,4.25) .. (-4,4.5);
 \draw[line width=2pt] (3,7) .. controls (2.5,6.75) .. (2,6.5);
 \draw[line width=2pt] (3,8) .. controls (2.5,8.25) .. (2,8.5);
 \draw[line width=2pt] (1,7) .. controls (0.5,6.75) .. (0,6.5);
 \draw[line width=2pt] (1,8) .. controls (0.5,8.25) .. (0,8.5);
 \draw[line width=2pt] (2,2.5) .. controls (2.5,2.75) .. (3,3);
 \draw[line width=2pt] (3,4) .. controls (2.5,4.25) .. (2,4.5);
 \draw[line width=2pt] (1,4) .. controls (0.5,4.25) .. (0,4.5);
 \draw[line width=2pt] (1,3) .. controls (0.5,2.75) .. (0,2.5);
 \draw[line width=2pt] (2,-1.5) .. controls (2.5,-1.25) .. (3,-1);
 \draw[line width=2pt] (3,0) .. controls (2.5,0.25) .. (2,0.5);
 \draw[line width=2pt] (1,0) .. controls (0.5,.25) .. (0,0.5);
 \draw[line width=2pt] (1,-1) .. controls (0.5,-1.25) .. (0,-1.5);
 \draw[line width=2pt] (8,7).. controls (8.5,6.75) .. (9,6.5);
 \draw[line width=2pt] (10,8) .. controls (9.5,8.25) .. (9,8.5);
 \draw[line width=2pt] (8,8).. controls (7.5,8.25) .. (7,8.5);
 \draw[line width=2pt] (10,7).. controls (10.5,6.75) .. (11,6.5);
 \draw[line width=2pt] (14,7) .. controls (14.5,6.75) .. (15,6.5);
 \draw[line width=2pt] (16,8).. controls (15.5,8.25) .. (15,8.5);
 \draw[line width=2pt] (14,8) .. controls (13.5,8.25) .. (13,8.5);
 \draw[line width=2pt] (16,7) .. controls (16.5,6.75) .. (17,6.5);
 \draw[line width=2pt] (8,0) .. controls (7.5,0.25) .. (7,0.5);
 \draw[line width=2pt] (8,-1)  .. controls (7.5,-1.25) .. (7,-1.5);
 \draw[line width=2pt] (9,-1.5) .. controls (9,-2) .. (9,-2.5);
 \draw[line width=2pt] (10,-1) .. controls (10.5,-1.25) .. (11,-1.5);
 \draw[line width=2pt] (10,0) .. controls (10.5,0.25) .. (11,0.5);
 \draw[line width=2pt] (9,0.5) .. controls (9,1) .. (9,1.5);
 \draw[line width=2pt] (14,0) .. controls (13.5,0.25) .. (13,0.5);
 \draw[line width=2pt] (14,-1) .. controls (13.5,-1.25) .. (13,-1.5);
 \draw[line width=2pt] (15,-1.5) .. controls (15,-2) .. (15,-2.5);
 \draw[line width=2pt] (16,-1) .. controls (16.5,-1.25) .. (17,-1.5);
 \draw[line width=2pt] (16,0)  .. controls (16.5,0.25) .. (17,0.5);
 \draw[line width=2pt] (15,0.5) .. controls (15,1) .. (15,1.5);
 \draw[line width=2pt] (-0.5,7) ..controls(-0.5,7)..node[pos=0.2,anchor=south]{$i)$}(-0.5,7);

 \draw[line width=2pt] (-0.5,3) .. controls (-0.5,3)..node[pos=0.2,anchor=south]{$ii)$} (-0.5,3);
 \draw[line width=2pt] (-0.5,-1) ..controls(-0.5,-1)..node[pos=0.2,anchor=south]{$iii)$}(-0.5,-1);

 \draw[line width=2pt] (-6.5,3) ..controls(-6.5,3)..node[pos=0.2,anchor=south]{$1)$}(-6.5,3);

 \draw[line width=2pt] (6.5,7) ..controls(6.5,7)..node[pos=0.2,anchor=south]{$2)$}(6.5,7);

 \draw[line width=2pt] (6.5,-1).. controls(6.5,-1)..node[pos=0.2,anchor=south]{$3)$}(6.5,-1);
 \draw[line width=2pt] (-1.5,3).. controls(-1.5,3)..node[pos=0.2,anchor=south]{$\Longrightarrow$}(-1.5,3);
 \draw[line width=2pt] (12,7)..controls(12,7)..node[pos=0.2,anchor=south]{$\Longrightarrow$}(12,7);
 \draw[line width=2pt] (12,-1).. controls(12,-1)..node[pos=0.2,anchor=south]{$\Longrightarrow$}(12,-1);
\end{tikzpicture} \]
\caption{g=6, $x_T(u)=x_T(v)$ , $x_{T'}(v)=x_{T'}(w)$ and
$x_{T''}(u')=x_{T''}(v')$  }
\end{figure}

Now assume that $g\geq 7$. We can find two pairs of
adjacent vertices of degree $2$ in $G$, because of $n_{2}(G)\geq 4$
and $G$ has a perfect matching. We suppose that ${u_{k},u_{k+1}}$
and ${u_{l},u_{l+1}}$ are these vertices. Suppose that $G^{'}$
obtained from $G$ by identifying three vertices
$u_{k-1},u_{k},u_{k+1}$ as one vertex $u_{k-1}$ and also by
identifying three vertices $u_{l},u_{l+1},u_{l+2}$ as one vertex
$u_{l+2}$, where $l\geq k+2$. Obviously, $G^{'}$ is a broken sun
graph with a perfect matching $M^{'}$ whose girth is $g-4$ and
$n_2(G')\equiv 0 (mod\,4)$. Thus, using induction hypothesis in
$G^{'}$ by removing $e=uv\notin M^{'}$, $x(u)=x(v)$. So $G^{'}$ has
an eigenvalue $2$ with the eigenvector $X=(x_{1},\ldots
,x_{k-1},x_{k+2},x_{k+3},\ldots ,x_{l-1},x_{l+2},\ldots
,x_{n})^{t}\in \mathbb {R}^{n-4}$ such that $x_{i}\in \{-r,r\}$. If
$l\geq k+3$, then we define the vector $Y=(y_{1},\ldots
,y_{n})^{t}\in \mathbb {R}^{n}$ as

$$y_{i}=
\begin{cases}
     x_{i}, & \mathrm{if}\,\, 1\leq i\leq k-1; \\
      x_{k+2}, & \mathrm{if}\,\, i=k; \\
-x_{k-1},& \mathrm{if}\,\,i=k+1; \\
-x_{i}, & \mathrm{if}\,\,k+2\leq i\leq l-1;\\
-x_{l+2}& \mathrm{if}\,\, i=l;\\
x_{l-1} & \mathrm{if}\,\,i=l+1;\\
x_{i}& \mathrm{if}\,\, l+2\leq i\leq n;
\end{cases}$$
also assign to each pendant vertex the negative value of its
neighbor. If $l=k+2$, then we define the vector $Y=(y_{1},\ldots
,y_{n})^{t}\in \mathbb {R}^{n}$ as
 $$y_{i}=
\begin{cases}
     x_{i}, & \mathrm{if}\,\, 1\leq i\leq k-1; \\
      x_{k+4}, & \mathrm{if}\,\, i=k; \\
-x_{k-1},& \mathrm{if}\,\,i=k+1; \\
-x_{k+4}, & \mathrm{if}\,\,i=k+2=l;\\
x_{k-1}& \mathrm{if}\,\, i=k+3=l+1;\\
x_{i}& \mathrm{if}\,\, k+4\leq i\leq n;
\end{cases}$$
also assign to each pendant vertex the negative value of its
neighbor. One may check that in both cases, the vector $Y$ satisfies
in Equation (\ref{one}). Therefore, $Y$ is an eigenvector of $L(G)$
corresponding to the eigenvalue $2$ such that $y_{i}\in \{-r,r\}$
and the proof is complete.
\end{proof}
\par In the follow, we wish to prove the corresponding of Theorem
\ref{seigenv} for any unicyclic graphs containing a perfect matching
for which the Theorem \ref{seigenv} plays as an induction basis.

\begin{theorem}\label{ueigenv}
Let $G=C(T_1,\ldots,T_g)$ be a unicyclic graph containing a perfect
matching which has Laplacian eigenvalue $2$. Then there exists the
eigenvector corresponding to the eigenvalue $2$ such as
$X=(x_1,\ldots ,x_n)^t$, such that $x_i\in \{-r,r\}$.
\end{theorem}

\begin{proof}
First note that, for broken sun graphs, the proof is clear by
Theorem \ref{seigenv}. So, let $|V(T_i)|\geq 3$, for some $i$,
$1\leq i\leq g$. We prove the theorem by induction on $n=|V(G)|$.
 Let $d(u,v_i)=max_{x\in V(T_i)}d(x,v_i)$, where $v_i$
is the root of $T_i$. Since $G$ has a perfect matching, $u$ is a
pendant vertex and its neighbor, say $v$, has degree $2$. Thus
$G=(G\setminus \{u,v\})\odot S_2$, where $S_2$ is a star on $2$
vertices. $L(G\setminus \{u,v\})$ has an eigenvalue $2$ because
$m_G(2)=m_{(G\setminus \{u,v\})}(2)$, by \cite[Theorem 2.5]{grone}.
So, by induction hypothesis, $X=(x_1,\ldots ,x_{n-2})^t$ is the
eigenvector of $L(G\setminus \{u,v\})$ corresponding to the
eigenvalue $2$ such that $x_i\in \{-r,r\}$ for all $i=1,\ldots
,n-2$. Let $w\neq u$ be the vertex that is a neighbor of $v$.
$Z=(z_1,\ldots ,z_{n})^t=(X,x(w),-x(w))^t$ is an eigenvector of
$L(G)$ corresponding to the eigenvalue $2$, where $z_i\in \{-r,r\}$
for all $i=1,\ldots ,n$. This is because that
\begin{align*}
&(d_{G\setminus \{u,v\}}(w)-2)x_{G\setminus \{u,v\}}(w)=\sum_{{v_{i}}\in N_{G\setminus \{u,v\}}(w)}x(v_{i})\\
&(d_{G\setminus \{u,v\}}(w)-2)x_{G\setminus \{u,v\}}(w)+x(w)=\sum_{{v_{i}}\in N_{G\setminus \{u,v\}}(w)}x(v_{i})+x(w)\\
&(d_G(w)-2)z_G(w)=\sum_{{v_{i}}\in N_G(w)}z(v_{i}),\\
\end{align*}

and for vertex $v$

$$  \quad \quad d_G(v)=2\Rightarrow
\begin{cases}
(d_G(v)-2)z(u)=0  \\
\sum_{v_j\in N_G(v)}z(v_j)=z(w)+z(u)=x(w)-x(w)=0  \\
\end{cases}.$$

Also for the vertex $u$

$$d_G(u)=1\Rightarrow
\begin{cases}
(d_G(u)-2)z(u)=-z(u)=x(w)  \\
\sum_{v_j\in
N_G(u)}z(v_j)=z(v)=x(w)  \\
\end{cases}.$$

By noting the fact that $d_G(p)=d_{G\setminus \{u,v\}}(p)$ for the
other vertices of $G$, we have $$(d_G(v)-2)z_G(v)=\sum_{v_j\in
N_G(v)}z(v_j),\, \mathrm{for}\, \, \mathrm{all}\, v\in V(G),$$ and
the proof is complete.
\end{proof}

\section{Laplacian eigenvalue $2$ of bicyclic graphs}
In this section we study the multiplicity of Laplacian eigenvalue
$2$ of bicyclic graphs.
 Let $g_1$ and $g_2$ be the girth of cycles of $C_1$ and $C_2$ respectively in bicyclic graph $G$.

\begin{lemma}\label{eqibi}
Let $G$ be a bicyclic graph and $\mu> 1$ be an integral eigenvalue
of $L(G)$. Then $m_{G}(\mu)\leq 3$.
\end{lemma}

\begin{proof}
By contrary, if $m_{G}(\mu)\geq 4$, then using Theorem
\ref{eqution}, for every unicyclic subgraph $G'$ of $G$ we have
$m_{G'}(\mu)\geq 3$. This contradicts \cite[Lamma~4]{akbari}
 and the result follows.
\end{proof}


\begin{theorem}\label{biodd}
Let $G$ be a bicyclic graph of odd order $n$. Then $m_{G}(2)\leq 2$.
In particular, if $g_{1},g_{2}\neq 0 (mod\,2)$ then $m_{G}(2)=0$.
\end{theorem}

\begin{proof}
By contrary, suppose that $m_{G}(2)\geq 3$. Let $C_1$ and $C_2$ be
two cycles of $G$. Let $G'=G-{e}$, where $e\in E(C_1)$. Then $G'$ is
a unicycle graph. So $m_{G'}(2)\geq 2$, by Theorem \ref{eqution}. If
$m_{G'}(2)> 2$, this contradicts \cite[Lamma~4]{akbari}. Thus
$m_{G'}(2)= 2$. Let $T$ be a spanning tree of $G^{'}$. Therefore
$L(T)$ has an eigenvalue 2, by Theorem \ref{eqution}. By applying
\cite[Theorem 2.1]{grone}, we conclude that $2\mid n$, a
contradiction. Therefore, $m_{G}(2)\leq 2$. Moreover, Theorem
\ref{component} implies that $\xi_{n-2}=\sum_{F\in
\mathfrak{F}_{2}}\gamma(F)$. Since $n$ is odd, for each $F\in
\mathfrak{F}_{2}$, the value of $\gamma(F)$ is even. So, $\xi_{n-2}$
is an even number. Thus, if $L(G)$ has an eigenvalue $2$, then
$4\mid \xi_{n-1}=n\tau(G)= 2ng_1g_2$ and hence $2 \mid ng_{1}g_{2}$.
Therefore, $2\mid g_{1}g_{2}$, a contradiction, and the proof is
complete.
\end{proof}

\begin{remark}\label{perfect}
Let $G=G_1\odot G_2$ be a bicyclic graph such that $G_1$ and $G_2$
containing a perfect matching. It is obvious that $G$ has a perfect
matching.
\end{remark}

\begin{theorem}\label{unicyclicbi}
Let $G_1=C(T_1,\ldots ,T_{g_1})$ and $G_2=C(T'_1,\ldots
,T^{'}_{g_2})$ be unicyclic graphs containing a perfect matching.
Let the one-edge connected graph $G=G_1\odot G_2$ has a prefect
matching $M$ and $L(G)$ has an eigenvalue $2$. Then $G-e$ has a
Laplacian eigenvalue $2$ such that $e\in C_{g_1}$ or $e\in C_{g_2}$
and $e\notin M$.
\end{theorem}

 \begin{proof}
Let $|V(G_1)|=n_1$ and $|V(G_2)|=n_2$. Without lose of generality,
we can assume that $e\in C_{g_1}$ and $\mu_k(G)=2$. So by Theorem
\ref{eqution}, we have, $$\mu_{k-1}(G)\geq \mu_{k-1}(G-e)\geq
\mu_{k}(G)\geq \mu_{k}(G-e)\Longrightarrow 2\geq \mu_{k}(G-e).$$
Now, let $e'\in C_{g_{2}}$. Then by Theorem \ref{eqution}, we have,
$$2\geq \mu_{k}(G-e)\geq \mu_{k}(G-e-e').$$ Assume that
$n=n_1+n_2=2k'$. Since $G-e-e'$ has a perfect matching,
$\mu_{k'}(G-e-e')=2$ and hence $2=\mu_{k'}(G-e-e')\geq
\mu_k(G-e-e')$ and $k\geq k'$ by Theorem \ref{eqution}. If $k=k'$,
then $2\geq \mu_k(G-e)\geq 2$ so $\mu_k(G-e)=2$ and the proof is
complete. On the other hand, if $k> k'$, then $\mu_{k}(G-e-e')\neq
2$. So we have, $$2=\mu_{k'}(G-e-e')\geq \mu_{k'+1}(G-e-e')\geq
\mu_{k-1}(G-e)\geq \mu_{k}(G)=2$$ and therefore
$\mu_{k'+1}(G-e-e')=2$. This is a contradiction, by
\cite[Theorem~2.1]{grone} and the result holds.
\end{proof}

\par As an immediate result we have.

\begin{corollary}
Let $G_1=C(T_1,\ldots ,T_{g_1})$ and $G_2=C(T'_1,\ldots
,T^{'}_{g_2})$ be unicyclic graphs containing a perfect matching and
$\mu_{k}(G)=2$. Then $\mu_{k}(G)=\mu_k(G-e)=\mu_k(G-e-e')=2$, where
$e\in C_{g_1}$ and $e'\in C_{g_2}$ and $\{e,e'\}\cap M=\emptyset$.
\end{corollary}

In the follow, we state the condition under which the bicyclic
graphs have a Laplacian eigenvalue $2$.

\begin{theorem}\label{twotree}
Let $G_{1}=C(T_{1},\ldots ,T_{g_{1}})$ and $G_{2}=C(T^{'}_{1},\ldots
,T^{'}_{g_{2}})$ be unicyclic graphs containing a perfect matching
which have a Laplacian eigenvalue $2$ and $G=G_1\odot G_2$ be a
bicyclic graph. Let $s_{1}$ and $s_{2}$ be the number of $T_{i}$ and
$T'_{j}$ of odd orders of $G_{1}$ and $G_{2}$, respectively. Then
$s_{1}\equiv s_{2}\equiv 0(mod\,4)$ if and only if $L(G)$ has an
eigenvalue $2$.
\end{theorem}

\begin{proof}
Assume $X$ and $Y$ are eigenvectors of $L(G_1)$ and $L(G_2)$
corresponding to the eigenvalue $2$, respectively. So vectors $X$
and $Y$ satisfy Equation (\ref{one}). Let $u$ and $v$ be two
vertices of $V(G_1)$ and $V(G_2)$ with $uv\in E(G)$. Now let
$X'=(X,\frac{x(u)}{y(v)}Y)$. We show that $X'$ satisfies Equation
(\ref{one}) for $\mu=2$. First note that $d_G(p)=d_{G_1}(p)$ for all
$p\in V(G_1)-\{u\}$ and $d_G(q)=d_{G_2}(q)$ for all $q\in
V(G_2)-\{v\}$. So Equation (\ref{one}) holds for all vertices
$V(G)-\{u,v\}$. Also,

$(d_G(u)-2)x'(u)=(d_{G_1}(u)-2)x'(u)+x'(u)=(d_{G_1}(u)-2)x(u)+x(u) $
\begin{align*}
&=\sum_{v_j\in N_{G_1}(u)}x(v_j)+x(u)=\sum_{v_j\in
N_{G_1}(u)}x'(v_j)+\frac{x(u)}{y(v)}y(v) \\
&=\sum_{v_j\in N_{G_1}(u)}x'(v_j)+x'(v)=\sum_{v_j\in N_G(u)}x'(v_j) \\
\end{align*}
and

$(d_G(v)-2)x'(v)=(d_{G_2}(v)-2)x'(v)+x'(v)=(d_{G_2}(v)-2)\frac{x(u)}{y(v)}y(v)+y(v)\frac{x(u)}{y(v)}$
\begin{align*}
&=\frac{x(u)}{y(v)}\sum_{v_j\in N_{G_2}(u)}y(v_j)+x(u)=\sum_{v_j\in N_{G_2}(v)}(y(v_j)\frac{x(u)}{y(v)})+x'(u) \\
&=\sum_{v_j\in N_{G_2}(u)}x'(v_j)+x'(u)=\sum_{v_j\in N_G(u)}x'(v_j) \\
\end{align*}
 Thus the proof of the `only if' part of the theorem is complete.

Conversely, assume that $L(G)$ has an eigenvalue $2$. Suppose $e=uv$
is a joining edge of $G$ with $u\in V(T_i)$ and $v\in V(T'_j)$. The
unicyclic graph $G-e_1$ has a Laplacin eigenvalue $2$, where $e_1\in
C_{g_{1}}$ or $e_1\in C_{g_{2}}$ and $e_1$ is not in the perfect
matching $M$ of $G$, by Theorem \ref{unicyclicbi}. Without lose of
generality, let $e_1\in C_{g_1}$. Then $s\equiv 0\ (mod\,4)$, where
$s$ is the number of trees of odd orders in $G-e_1=C(T'_1,\ldots
,T'_{j-1},T'_j\cup \{e\}\cup \{G_1-e_1\},T'_{j+1},\ldots
,T'_{g_2})$, by \cite[Theorem 9]{akbari}. If $|V(T'_j)|$ is an even
number, then $T'_j\cup \{e\}\cup \{G_1-e_1\}$ is an even number. So
the trees of odd orders in $G_2$ are the same as the trees of odd
orders in $G-e_1$ and hence $s_2\equiv 0\ (mod\,4)$. If $|V(T'_j)|$
is an odd number, then $T'_j\cup \{e\}\cup \{G_1-e_1\}$ is an odd
number.  So the trees of odd orders in $G_2$ are $T'_j$ and all
trees of odd orders in $G-e_1$ except $T'_j\cup \{e\}\cup
\{G_1-e_1\}$ (see figure $5$). Therefore, $ s_2\equiv 0\ (mod\,4)$
and this completes the proof.

\begin{figure}[!ht]
  \[\begin{tikzpicture}[scale=.4,thin]
  \draw (-7,0) circle (2cm)[];
 \vertex[fill] (v1) at (-5,2.41)[label=north west:$T_i$] {};
 \vertex[fill] (v2) at (-3,1.4) [label=south east:$u$] {};
 \vertex[fill] (v3) at (-5,0)[] {};
 \vertex[fill] (v7) at (-7,-2) [] {};
 \vertex[fill] (v8) at (-5.5,-1.4) [] {};
 \draw (7,0) circle (2cm);
 \vertex[fill] (v4) at (5,0) [] {};
 \vertex[fill](v5) at (3,1.41) [label=south west:$v$] {};
 \vertex[fill](v6) at (5,2.41) [label=north east:$T'_j$] {};
 \vertex[fill] (v9) at (7,-2) [] {};
 \vertex[fill] (v10) at (5.5,-1.4) [] {};
 \path
 (v1) edge (v2)
 (v3) edge (v2)
 (v5) edge (v6)
 (v4) edge (v5);
\draw[dashed] (4.5,1) [partial ellipse=21:370:1.9cm and 2.2cm];
\draw[dashed] (-4.5,1) [partial ellipse=21:370:1.9cm and 2.2cm];
 \draw[line width=2pt] (-3,1.41) .. controls (0,1.41) .. (3,1.41);
 \draw[line width=0pt] (-5.5,-1.4) .. controls (-6.5,-2) ..node[pos=0.5,anchor=south ]{$e_1$} (-7,-2);
 \draw[line width=0pt] (5.5,-1.4) .. controls (6.5,-2) ..node[pos=0.5,anchor=south ]{$e_2$} (7,-2);
\draw (-4,-1.5) ++(-2.5,-1.5) node[anchor=north]{$G_1$} (-1,-1.5);
\draw (2.5,-1.5) ++(4,-1.5) node[anchor=north]{$G_2$} (5.5,-1.5);
\end{tikzpicture} \]
\caption{$G=G_1\odot G_2$}
\end{figure}
\end{proof}
\par As an immediate result from  Theorems \ref{seigenv} and
\ref{twotree}, we have.
\begin{corollary}\label{sung}
Let $G_1=C(T_1,\ldots ,T_{g_1})$ and $G_2=C(T'_1,\ldots
,T^{'}_{g_2})$ be unicyclic graphs containing a perfect matching and
$\mu_{k}(G)=2$. Then $\mu_{k}(G)=\mu_k(G-e)=\mu_k(G-e-e')=2$, where
$e\in C_{g_1}$ and $e'\in C_{g_2}$ and $\{e,e'\}\cap M=\emptyset$.
\end{corollary}
Let $G_1$ and $G_2$ be two unicyclic graphs. Assume $L(G_1)$
($L(G_2)$) has an eigenvalue $2$ and $G=G_1\odot G_2$. if $L(G)$ has
an eigenvalue $2$, then it can not be concluded that $L(G_2)$
($L(G_1)$) has an eigenvalue $2$, as the following example shows:

\begin{example}
Let $G_1$ and $G_2$ be unicyclic graphs in below figure. $L(G_1)$
and $L(G)$ have an eigenvalue $2$ but $L(G_2)$ has no an eigenvalue
$2$.

\begin{figure}[!ht]
  \[\begin{tikzpicture}[scale=.4,thin]
 \vertex[fill] (v2) at (-3,1.4) [label=south:$u$] {};
 \vertex[fill] (v2) at (-3,1.4) [label=north:$0$] {};
 \vertex[fill] (v3) at (-5,1.4)[label=north:$0$] {};
 \vertex[fill] (v4) at (-7,1.4)[label=north:$-1$] {};
 \vertex[fill] (v5) at (-5,-0.6)[label=south:$1$] {};
 \vertex[fill] (v6) at (-7,-0.6)[label=south:$0$] {};
 \vertex[fill] (v7) at (5,0) [label=south:$0$] {};
 \vertex[fill](v8) at (3,1.41) [label=north:$0$] {};
 \vertex[fill](v8) at (3,1.41) [label=south:$v$] {};
 \vertex[fill](v9) at (1,0)[label=south:$0$]{};
 \path
 (v2) edge (v3)
 (v3) edge (v4)
 (v4) edge (v6)
 (v3) edge (v5)
 (v5) edge (v6)
 (v8) edge (v7)
 (v8) edge (v9)
 (v7) edge (v9);

 \draw[line width=3pt] (-3,1.41) .. controls (0,1.41) .. (3,1.41);
\draw (-3,-1) ++(-1.5,-1) node[anchor=north]{$G_1$} (0,-1); \draw
(1,-1) ++(2.5,-1) node[anchor=north]{$G_2$} (4,-1);
\end{tikzpicture} \]
\caption{$G=G_1\odot G_2$}
\end{figure}
\end{example}
Here we establish some condition on degrees of vertices of type of
bicyclic graphs for which such graph has Laplacian eigenvalue $2$.

\begin{theorem}\label{nopermat}
Let $G_1$ and $G_2$ be broken sun graphs of orders $n_1$ and $n_2$
with no perfect matching. If $g_1\equiv g_2\equiv 0\ (mod\,4)$ and
there are odd numbers of vertices of degree $2$ between any pair of
consecutive vertices of degree $3$, then $L(G=G_1\odot G_2)$ has an
eigenvalue $2$.
\end{theorem}

\begin{proof}
Assume that $g_1\equiv g_2\equiv 0\ (mod\,4)$ and there are odd
numbers of vertices of degree $2$ between any pair of consecutive
vertices of degree $3$ in $G_1$ and $G_2$, therefore $L(G_1)$ and
$L(G_2)$ have an eigenvalue $2$, by \cite[Theorem~10]{akbari}. Let
the edge of $G$ joining $G_1$ and $G_2$ be $e=uv$, where $v\in
V(G_1)$ and $u\in V(G_2)$. We can assign $\{-1,0,1\}$ to the
vertices of $C_{g_1}$ and $C_{g_2}$, by the pattern $0,1,0,-1$
consecutively starting with a vertex of degree $3$, and assign to
each pendant vertex the negative of value of its neighbor to obtain
eigenvectors $X$ and $Y$ of $L(G_1)$ and $L(G_2)$ corresponding to
the eigenvalues $2$, respectively. If $u$ and $v$ are two pendant
vertices or two vertices of degree $3$ or one of them is a pendant
vertex and the other is  of degree $3$, then $Z=(X,Y)$ is an
eigenvector of $L(G)$ corresponding to the eigenvalue $2$ (note that
Equation (\ref{one}) is satisfied). If $u$ is a vertex of degree one
or degree $3$ and $v$ is a vertex of degree $2$, so
$Z=(X,\overrightarrow{0})$ is an eigenvector of $L(G)$ corresponding
to the eigenvalue $2$ (note that Equation (\ref{one}) is satisfied).
If $u$ and $v$ are two vertices of degree $2$, thus if $x(u)=y(v)$,
then $Z=(X,Y)$. If $x(u)=0$ and $y(v)\neq 0$, so
$Z=(X,\overrightarrow{0})$. If $x(u)\neq y(v)$ and both of them are
other than $0$, hence $Z=(X,-Y)$. Therefore $Z$ satisfies in
Equation (\ref{one}) for $\mu=2$ and $L(G)$ has an eigenvalue $2$.
Therefore the result follows.
\end{proof}

\begin{theorem}\label{onepermat}
Let $G_1$ be a broken sun graph of orders $n_1$ with no perfect
matching and $G_2$ be a unicyclic graph of order $n_2$ with a
perfect matching. If $L(G_1)$ and $L(G_2)$ have an eigenvalue $2$,
then $L(G_1\odot G_2)$ has an eigenvalue $2$.
\end{theorem}

\begin{proof}
It is proved like Theorem \ref{nopermat} by the similar
method.
\end{proof}

    \bibliographystyle{siamplain}
$$$$
    \bibliography{references}

\end{document}